\documentclass[12pt]{article}
\usepackage{amssymb,amsmath,mathrsfs,amsthm}
\usepackage{hyperref}
\usepackage{graphicx}
\usepackage{color}
\usepackage[OT2,OT1]{fontenc}
\usepackage{upquote}
\usepackage{authblk}
\usepackage{float}
\usepackage[numbers,sort&compress]{natbib}

\newtheorem{theorem}{Theorem}[section]
\newtheorem{lemma}[theorem]{Lemma}

\begin{document}

\title{\Large\bf Unconditional applicability of \\
Lehmer's measure to the two-term \\
Machin-like formula for $\pi$}

\bigskip
\author[1, 2, 3]{\small Sanjar M. Abrarov}
\author[2, 3, 4]{\small Rehan Siddiqui}
\author[3, 4]{\small Rajinder K. Jagpal}
\author[1, 2, 4]{\small \\ Brendan M. Quine}

\affil[1]{\scriptsize Thoth Technology Inc., Algonquin Radio Observatory, Achray Rd, RR6, Pembroke, Canada, K8A 6W7
\normalsize}
\affil[2]{\scriptsize Dept. Earth and Space Science and Engineering, York University, 4700 Keele St., Canada, M3J 1P3 \normalsize}
\affil[3]{\scriptsize Epic College of Technology, 5670 McAdam Rd., Mississauga, Canada, L4Z 1T2 \normalsize}
\affil[4]{\scriptsize Dept. Physics and Astronomy, York University, 4700 Keele St., Toronto, Canada, M3J 1P3 \normalsize}

\date{June 8, 2021}
\maketitle

\begin{abstract}
Lehmer defined a measure
$$
\mu=\sum\limits_{j=1}^J\frac{1}{\log_{10}\left(\left|\beta_j\right|\right)},
$$
where the $\beta_j$ may be either integers or rational numbers in a Machin-like formula for $\pi$. When the $\beta_j$ are integers, Lehmer's measure can be used to determine the computational efficiency of the given Machin-like formula for $\pi$. However, because the computations are complicated, it is unclear if Lehmer’s measure applies when one or more of the $\beta_j$ are rational. In this article, we develop a new algorithm for a two-term Machin-like formula for $\pi$ as an example of the unconditional applicability of Lehmer's measure. This approach does not involve any irrational numbers and may allow calculating $\pi$ rapidly by the Newton–Raphson iteration method for the tangent function.
\vspace{0.25cm}
\vspace{0.25cm}
\end{abstract}

\section{Introduction}

In 1706, the English astronomer and mathematician John Machin discovered a two-term formula for $\pi$ \cite{Beckmann1971, Berggren2004, Borwein2008}
\begin{equation}\label{eq_1}
\frac{\pi }{4} = 4\arctan \left( {\frac{1}{5}} \right) - \arctan \left( {\frac{1}{{239}}} \right),
\end{equation}
that was later named in his honor. This formula for $\pi$ appeared to be more efficient than any others known by that time. In particular, due to the relatively rapid convergence of the right-hand side of \eqref{eq_1}, he was able to calculate $100$ decimal digits of $\pi$ \cite{Beckmann1971}. Nowadays, identities of the form
\begin{equation}\label{eq_2}
\frac{\pi }{4} = \sum\limits_{j = 1}^J {{\alpha _j}\arctan \left( {\frac{1}{{{\beta _j}}}} \right)},
\end{equation}
where ${\alpha _j}$ and ${\beta _j}$ are either integers or rational numbers, are called the Machin-like formulas for $\pi$. Consequently, a two-term Machin-like formula for $\pi$ is given by
\begin{equation}\label{eq_3}
\frac{\pi }{4} = {\alpha _1}\arctan \left( {\frac{1}{{{\beta _1}}}} \right) + {\alpha _2}\arctan \left( {\frac{1}{{{\beta _2}}}} \right).
\end{equation}
If in equation \eqref{eq_3} the constants ${\alpha _1}$ and ${\beta _1}$ are some positive integers and ${\alpha _2} = 1$, then the unknown value ${\beta _2}$ can be found as \cite{Abrarov2017a, Abrarov2017b}
\small
\begin{equation}\label{eq_4}
\frac{1}{{{\beta _2}}} = \frac{2}{{{{\left[ {\left( {{\beta _1} + i} \right)/\left( {{\beta _1} - i} \right)} \right]}^{{\alpha _1}}} + i}} + i \Leftrightarrow \beta_2 = \frac{2}{{{{\left[ {\left( {{\beta _1} + i} \right)/\left( {{\beta _1} - i} \right)} \right]}^{{\alpha_1}}} - i}} - i.
\end{equation}
\normalsize
Furthermore, since we assumed that ${\alpha_1}$ and ${\beta_1}$ are positive integers, from equation \eqref{eq_4} it immediately follows that ${\beta_2}$ must be either an integer or a rational number.

In 1938, Lehmer \cite{Lehmer1938} introduced a measure (see also \cite{Wetherfield1996})
\begin{equation}\label{eq_5}
\mu  = \sum\limits_{j = 1}^J {\frac{1}{{{{\log }_{10}}\left( {\left| {{\beta _j}} \right|} \right)}}},
\end{equation}
showing how much computational effort is required for a specific Machin-like formula for $\pi$. In particular, when $\mu$ is small, then less computational effort is required and, consequently, the computational efficiency of this formula is higher. Lehmer's measure is smaller if there are fewer summation terms $J$ and the constants $\beta_j$ are larger in magnitude. For more efficient computation, the constants $\beta_j$ should be larger by absolute value, since it is easier to approximate the arctangent function as its argument tends to zero (see \cite{Wetherfield1996} for more details).

It is also important to emphasize that in the same paper \cite{Lehmer1938} Lehmer presented a few Machin-like formulas where some of the ${\beta _j}$ are not integers but rational numbers. This signifies that Lehmer assumed that his measure \eqref{eq_5} remains valid whether the $\beta_j$ are integers or rational numbers.

In 2002 Kanada \cite{Calcut2009}, using the following self-checking pair of Machin-like formulas for $\pi$,
\small
\[
\frac{\pi }{4} = 44\arctan \left( {\frac{1}{{57}}} \right) + 7\arctan \left( {\frac{1}{{239}}} \right) - 12\arctan \left( {\frac{1}{{682}}} \right) + 24\arctan \left( {\frac{1}{{12943}}} \right)
\]
\normalsize
and
\small
\[
\frac{\pi }{4} = 12\arctan \left( {\frac{1}{{49}}} \right) + 32\arctan \left( {\frac{1}{{57}}} \right) - 5\arctan \left( {\frac{1}{{239}}} \right) + 12\arctan \left( {\frac{1}{{110443}}} \right)
\]
\normalsize
computed more than $1$ trillion digits of $\pi$. These two examples show that the Machin-like formulas have colossal potential in the computation of decimal digits of $\pi$.

In 1997, Chien-Lih showed a remarkable formula \cite{Chien-Lih1997}
\small
\[
\begin{aligned}
\frac{\pi }{4} =& 183\arctan \left( {\frac{1}{{239}}} \right) + 32\arctan \left( {\frac{1}{{1023}}} \right) - 68\arctan\left( {\frac{1}{{5832}}} \right) \\
&+ 12\arctan \left( {\frac{1}{{110443}}} \right) - 12\arctan \left( {\frac{1}{{4841182}}} \right) - 100\arctan \left( {\frac{1}{{6826318}}} \right)
\end{aligned}
\]
\normalsize
with $\mu  \approx 1.51244$. According to Weisstein \cite{Weisstein2020}, this Lehmer's measure is the smallest known value for the $\beta_j$ consisting of integers only. Later Chien-Lih \cite{Chien-Lih2004} showed how Lehmer's measure can be reduced even further by using an Euler type of identity in an iteration for generating the two-term Machin-like formulas like \eqref{eq_3} such that $\beta_1$ and $\beta_2$ are rational numbers.

In \cite{Abrarov2018a}, we derived the following simple identity (see also \cite{OEIS2017})
\begin{equation}\label{eq_6}
\frac{\pi }{4} = {2^{k - 1}}\arctan \left( {\frac{{\sqrt {2 - {c_{k - 1}}} }}{{{c_k}}}} \right), \qquad	k = \left\{ {2,3,4, \ldots } \right\},
\end{equation}
where $c_1=0$ and $c_k=\sqrt{2+c_{k-1}}$, and described how using this identity, another efficient method for generating the two-term Machin-like formula for$\;\pi$
\begin{equation}\label{eq_7}
\frac{\pi }{4} = {2^{k - 1}}\arctan \left( {\frac{1}{{{\beta _1}}}} \right) + \arctan \left( {\frac{1}{{{\beta _2}}}} \right),
\end{equation}
with small Lehmer's measure can be developed. In this approach, the constant ${\beta _1}$ can be chosen as a positive integer such that
\begin{equation}\label{eq_8}
{\beta _1} = \left\lfloor {\frac{{{c_k}}}{{\sqrt {2 - {c_{k - 1}}} }}} \right\rfloor
\end{equation}
and, in accordance with equation \eqref{eq_4}, the constant $\beta_2$ in equation \eqref{eq_7} can be found from
\begin{equation}\label{eq_9}
{\beta _2} = \frac{2}{{{{\left[ {\left( {{\beta _1} + i} \right)/\left( {{\beta _1} - i} \right)} \right]}^{{2^{k-1}}}} - i}} - i.
\end{equation}

It is not reasonable to solve equation \eqref{eq_9} directly to determine the rational number ${\beta _2}$, as its solution becomes tremendously difficult with increasing integer $k$. However, this problem can be effectively resolved by using a very simple two-step iteration procedure discussed in the next section. Therefore, our approach in generating the two-term Machin-like formula \eqref{eq_7} for $\pi$  with small Lehmer's measure is much easier than Chien-Lih's method \cite{Chien-Lih2004}.

Wetherfield \cite{Wetherfield1996} provides a detailed explanation clarifying the significance of Lehmer's measure that shows how much computation is required for a given Machin-like formula for $\pi$ when all the constants ${\beta _j} \in \mathbb{Z}$. However, it is unclear if this paradigm is also applicable when at least one number ${\beta _j}$ is rational. More specifically, the problem that occurs in computing the two-term Machin-like formula for  in Chien-Lih’s \cite{Chien-Lih2004} and our \cite{Abrarov2017a} iteration methods is related to the rapidly growing number of digits in the numerators and denominators of ${\beta _1}$ or ${\beta _2}$. This occurs simultaneously with an attempt to decrease Lehmer's measure. As a result, the subsequent exponentiation in conventional algorithms makes computing the decimal digits of $\pi$ inefficient. Therefore, the applicability of Lehmer's measure for a Machin-like formula for $\pi$ for the case $\beta_j \in \mathbb{Q}$ is questionable. For example, Lehmer's measure may be small, say less than 1. This means that less computational work is needed to calculate $\pi$. However, due to the large number of digits in the numerators and denominators in $\beta_1$ or $\beta_2$, a computer performs more intense arithmetic operations that make the runtime significantly longer. Consequently, we ask, Is Lehmer's measure still applicable when at least one constant from the set $\left\{\beta_j\right\}$ is not an integer but a rational number?

Motivated by an interesting paper in regard to equation \eqref{eq_7} that was recently published \cite{Wolfram2020}, we further develop our previous work \cite{Abrarov2017a, Abrarov2017b}. In this article, we propose a new algorithm showing how unconditional applicability of Lehmer's measure for the two-term Machin-like formula \eqref{eq_7} for  can be achieved. We also describe how linear and quadratic convergence to $\pi$ can be implemented.

\section{Preliminaries}

As mentioned, the number of summation terms $J$ in equation \eqref{eq_2} should be reduced in order to minimize Lehmer's measure. Since at $J=1$ there is only one Machin-like formula for $\pi$,
\begin{equation}\label{eq_10}
\frac{\pi }{4} = \arctan \left( 1 \right),
\end{equation}
we consider the case $J=2$. We attempted to find a method that can be used to generate a two-term Machin-like formula for $\pi$ with small Lehmer's measure. Equation \eqref{eq_7} provides an efficient way to do this.

In fact, the original Machin formula \eqref{eq_1} for $\pi$ appears quite naturally from the equations \eqref{eq_7}, \eqref{eq_8} and \eqref{eq_9} at $k=3$. Specifically, equation \eqref{eq_8} provides
\[
{\beta _1} = \left\lfloor {\frac{{{c_3}}}{{\sqrt {2 - {c_{3 - 1}}} }}} \right\rfloor  = \left\lfloor {\frac{{\sqrt {2 + \sqrt {2 + \sqrt 2 } } }}{{\sqrt {2 - \sqrt {2 + \sqrt 2 } } }}} \right\rfloor  = 5.
\]
Substituting into equation \eqref{eq_9} results in
\[
\frac{1}{{{\beta _2}}} = \frac{2}{{{{\left[ {\left( {5 + i} \right)/\left( {5 - i} \right)} \right]}^{{2^{3 - 1}}}} + i}} + i =  - \frac{1}{{239}}.
\]
Consequently, at $k = 3$ we get the constants ${2^{k - 1}} = 4$, ${\beta _1} = 5$ and ${\beta _2} =  - 239$ in equation \eqref{eq_7}. Since the arctangent function is odd ($\arctan\left(-x\right)=-\arctan\left(x\right)$), the constants for equation \eqref{eq_3} can be rearranged as $\alpha_1 = 4$, $\alpha_2 = -1$, ${\beta _1} = 5$ and ${\beta_2} = 239$. This corresponds to the original Machin formula \eqref{eq_1} for $\pi$.

\begin{theorem}

There are only four possible cases for a two-term Machin-like formula \eqref{eq_3} for $\pi$ when all four constants ${\alpha _1}$, ${\alpha _2}$, ${\beta _1}$ and ${\beta _2}$ are integers:
\newpage
\small
\begin{verbatim}
Text@Grid[{
   {"",Subscript[\[Alpha],1],Subscript[\[Alpha],2], 
    Subscript[\[Beta],1],Subscript[\[Beta],2]},
   {"Machin",4,-1,5,239},
   {"Euler",1,1,2,3},
   {"Hermann",2,-1,2,7},
   {"Hutton",2,1,3,7}
   },Frame->All,Alignment->Right]
\end{verbatim}
\normalsize

\begin{table}[H]
\begin{tabular}{| l | l | l | l | l |}
\hline
       &$\alpha_1$ & $\alpha_2$ & $\beta _1$ & $\beta_2$ \\ 
\hline
Machin  & 4         & -1         & 5          & 239 \\
\hline
Euler   & 1         &  1         & 2          & 3   \\
\hline
Hermann & 2         & -1         & 2          & 7   \\
\hline
Hutton  & 2         &  1         & 3          & 7   \\
\hline
\end{tabular}
\end{table}

\end{theorem}
\indent The proof for Theorem 2.1 can be found in \cite{Weisstein2020}.

\begin{lemma}
If in equation \eqref{eq_7} ${\beta _1} \in \mathbb{Z}$, then ${\beta _2} \notin \mathbb{Z}$ at any integer $k > 3$.
\end{lemma}

\begin{proof}
As we can see from the four cases given in Theorem 2.1, the largest possible value for ${\alpha _1}=2^{k-1}$ is $4$ and it occurs at $k = 3$ (see the example above). Therefore, for any integer ${\beta _1}$ at integer $k > 3$, the constant ${\beta _2}$ in the equation \eqref{eq_7} cannot be an integer.
\end{proof}

\begin{theorem}
$$
\mathop {\lim }\limits_{k \to \infty } {c_k} = \mathop {\lim }\limits_{k \to \infty } \underbrace {\sqrt {2 + \sqrt {2 + \sqrt {2 +  \cdots  + \sqrt 2 } } } }_{k\,\,{\text{square}}\,\,{\text{roots}}} = 2.
$$
\end{theorem}
\begin{proof}
This is the simplest kind of Ramanujan nested radical and its proof is straightforward. Let $X=\mathop{\lim}\limits_{k\to\infty}{c_k}$. Then
$$
X= \mathop {\lim }\limits_{k \to \infty } \sqrt {2 + {c_{k - 1}}}  = \sqrt {2 + \mathop {\lim }\limits_{k \to \infty } {c_{k - 1}}}  = \sqrt {2 + \mathop {\lim }\limits_{k \to \infty } {c_k}}.
$$
Squaring both sides,
$$
{X^2}=2 + \mathop {\lim }\limits_{k \to \infty } {c_k} \Leftrightarrow {X^2}=2 + X.
$$
and solving results in two possible solutions, ${X_1} =  - 1$ and ${X_2} = 2$. Since for any positive index $k$ the value ${c_k}$ is always positive, we exclude the solution ${X_1} =  - 1$.
\end{proof}

\begin{lemma}
$$
\mathop {\lim }\limits_{k \to \infty } \sqrt {2 - {c_{k - 1}}}  = \mathop {\lim }\limits_{k \to \infty } \sqrt {2 - \underbrace {\sqrt {2 + \sqrt {2 +  \cdots  + \sqrt 2 } } }_{k - 1\,\,{\text{square\,\,roots}}}}  = 0.
$$
\end{lemma}

\begin{proof}
The proof follows immediately from Theorem 2.3 since
$$
\mathop {\lim }\limits_{k \to \infty } \sqrt {2 - {c_{k - 1}}}  = \sqrt {2 - \mathop {\lim }\limits_{k \to \infty } {c_{k-1}}}  = \sqrt {2 - \mathop {\lim }\limits_{k \to \infty } {c_k}}  = \sqrt {2 - 2}  = 0.
$$
\end{proof}

\begin{lemma}
$$
\mathop {\lim }\limits_{k \to \infty } \frac{{\left( {\frac{{{c_k}}}{{\sqrt {2 - {c_{k - 1}}} }}} \right)}}{{{\beta _1}}} = 1.
$$
\end{lemma}

\begin{proof}
Using equation \eqref{eq_8} that defines ${\beta _1}$ by the floor function, the limit can be rewritten as
\begin{equation}\label{eq_11}
\mathop {\lim }\limits_{k \to \infty } \frac{{\left( {\frac{{{c_k}}}{{\sqrt {2 - {c_{k - 1}}} }}} \right)}}{{\left\lfloor {\frac{{{c_k}}}{{\sqrt {2 - {c_{k - 1}}} }}} \right\rfloor }} = 1.
\end{equation}
By definition, the fractional part given by the difference
$$
{\rm{frac}}\left( {\frac{{{c_k}}}{{\sqrt {2 - {c_{k - 1}}} }}} \right) = \left(\frac{{{c_k}}}{{\sqrt {2 - {c_{k - 1}}} }}\right) - \left\lfloor {\frac{{{c_k}}}{{\sqrt {2 - {c_{k - 1}}} }}} \right\rfloor,
$$
is positive and cannot be greater than 1. Therefore, the limit \eqref{eq_11} can be rewritten in the form
$$
 \mathop {\lim }\limits_{k \to \infty } \frac{{\left( {\frac{{{c_k}}}{{\sqrt {2 - {c_{k - 1}}} }}} \right) }}{{\left\lfloor {\frac{{{c_k}}}{{\sqrt {2 - {c_{k - 1}}} }}} \right\rfloor }}=\mathop {\lim }\limits_{k \to \infty } \frac{{\left\lfloor {\frac{{{c_k}}}{{\sqrt {2 - {c_{k - 1}}} }}} \right\rfloor  + {\text{frac}}\left( {\frac{{{c_k}}}{{\sqrt {2 - {c_{k - 1}}} }}} \right)}}{{\left\lfloor {\frac{{{c_k}}}{{\sqrt {2 - {c_{k - 1}}} }}} \right\rfloor }} = \mathop {\lim }\limits_{k \to \infty } \frac{{\left\lfloor {\frac{{{c_k}}}{{\sqrt {2 - {c_{k - 1}}} }}} \right\rfloor }}{{\left\lfloor {\frac{{{c_k}}}{{\sqrt {2 - {c_{k - 1}}} }}} \right\rfloor }} + 0 = 1.
$$
\end{proof}

\begin{lemma}
We have that
\begin{equation}\label{eq_12}
\lim_{k\to\infty}\left|\beta_2\right|=\infty.
\end{equation}
\end{lemma}

\begin{proof}
Equation \eqref{eq_9} is too hard to work with directly, so we start with the limit \eqref{eq_11} from Lemma 2.5. Since the limit is 1, the limit of the ratio of the reciprocals of its numerator and denominator must also be 1,
\begin{equation}\label{eq_13}
\mathop {\lim }\limits_{k \to \infty } \frac{{{{\left( {\frac{{{c_k}}}{{\sqrt {2 - {c_{k - 1}}} }}} \right)}^{ - 1}}}}{{\frac{1}{{{\beta _1}}}}} = \mathop {\lim }\limits_{k \to \infty } \frac{{\left( {\frac{{\sqrt {2 - {c_{k - 1}}} }}{{{c_k}}}} \right)}}{{\frac{1}{{{\beta _1}}}}} = 1.
\end{equation}
From Theorem 2.3 and Lemma 2.4,
$$
\mathop {\lim }\limits_{k \to \infty } {\frac{{\sqrt {2 - {c_{k - 1}}} }}{{{c_k}}}} = \frac{{\mathop {\lim }\limits_{k \to \infty } \sqrt {2 - {c_{k - 1}}} }}{{\mathop {\lim }\limits_{k \to \infty } {c_k}}} = \frac{0}{2} = 0
$$
and
\begin{equation}\label{eq_14}
\mathop {\lim }\limits_{k \to \infty } \frac{1}{{{\beta _1}}} = \mathop {\lim }\limits_{k \to \infty } \frac{1}{{\left\lfloor {\frac{{{c_k}}}{{\sqrt {2 - {c_{k - 1}}} }}} \right\rfloor }} = \frac{1}{{\left\lfloor {\frac{{\mathop {\lim }\limits_{k \to \infty } {c_k}}}{{\mathop {\lim }\limits_{k \to \infty } \sqrt {2 - {c_{k - 1}}} }}} \right\rfloor }} = \frac{1}{\infty} = 0.
\end{equation}
Since both the numerator and denominator in \eqref{eq_13} tend to zero as $k$ tends to infinity and since $\arctan\left(x\right)\to 0$ as $x\to 0$, we can rewrite \eqref{eq_13} as
$$
\mathop {\lim }\limits_{k \to \infty } \frac{{\arctan \left( {\frac{{\sqrt {2 - {c_{k - 1}}} }}{{{c_k}}}} \right)}}{{\arctan \left( {\frac{1}{{{\beta _1}}}} \right)}} = 1
$$
or
\begin{equation}\label{eq_15}
\mathop {\lim }\limits_{k \to \infty } \frac{{{2^{k - 1}}\arctan \left( {\frac{{\sqrt {2 - {c_{k - 1}}} }}{{{c_k}}}} \right)}}{{{2^{k - 1}}\arctan \left( {\frac{1}{{{\beta _1}}}} \right)}} = 1.
\end{equation}
Since equation \eqref{eq_6} is valid for an arbitrarily large integer $k$, \eqref{eq_15} implies that
\begin{equation}\label{eq_16}
\mathop {\lim }\limits_{k \to \infty } {2^{k - 1}}\arctan \left( {\frac{{\sqrt {2 - {c_{k - 1}}} }}{{{c_k}}}} \right) = \mathop {\lim }\limits_{k \to \infty } {2^{k - 1}}\arctan \left( {\frac{1}{{{\beta _1}}}} \right) = \frac{\pi }{4}.
\end{equation}
However, we also have
\begin{equation}\label{eq_17}
\mathop {\lim }\limits_{k \to \infty } \left[ {{2^{k - 1}}\arctan \left( {\frac{1}{{{\beta _1}}}} \right) + \arctan \left( {\frac{1}{{{\beta _2}}}} \right)} \right] = \frac{\pi }{4}
\end{equation}
since equation \eqref{eq_7} is also valid at an arbitrarily large integer $k$. Comparing the limits \eqref{eq_16} and \eqref{eq_17}, we get
\small
\begin{equation}\label{eq_18}
\mathop {\lim }\limits_{k \to \infty } {2^{k - 1}}\arctan \left( {\frac{1}{{{\beta _1}}}} \right) = \mathop {\lim }\limits_{k \to \infty } \left[ {{2^{k - 1}}\arctan \left( {\frac{1}{{{\beta _1}}}} \right) + \arctan \left( {\frac{1}{{{\beta _2}}}} \right)} \right] = \frac{\pi }{4}.
\end{equation}
\normalsize
Equation \eqref{eq_18} is valid if and only if
$$
\mathop {\lim }\limits_{k \to \infty } \arctan \left( {\frac{1}{{{\beta _2}}}} \right) = 0.
$$
so $\lim_{k\to\infty}\left|\beta_1\right|=\infty$, which is (12).

\end{proof}

\begin{lemma}
The Lehmer's measure \eqref{eq_5} may be vanishingly small.
\end{lemma}

\begin{proof}
The limit \eqref{eq_14} implies that
\begin{equation}\label{eq_19}
\mathop {\lim }\limits_{k \to \infty } {\beta _1} = \infty.
\end{equation}
From \eqref{eq_12} and \eqref{eq_19} we conclude that as $k\to\infty$, both $\beta_1$ and $\beta_2$ tend to infinity. Thus, according to equation \eqref{eq_5}, Lehmer's measure $\mu$ for the two-term Machin-like formula \eqref{eq_7} for $\pi$ tends to zero with increasing $k$.
\end{proof}

\begin{theorem}
If equation \eqref{eq_8} holds, then the constant  is always negative.
\end{theorem}

\begin{proof}
There is only one single-term Machin-like formula \eqref{eq_10} for $\pi$ such that the constants $\alpha_1$ and $\beta_1$ are both integers. Therefore, in equation \eqref{eq_6}, at any integer $k\ge 2$ the argument of the arctangent function $\sqrt {2 - {c_{k - 1}}} /{c_k}$ cannot be represented as the reciprocal of an integer; that is, ${c_k}/\sqrt {2 - {c_{k - 1}}}$ is not an integer. Therefore,
$$
{\rm{frac}}\left( {\frac{{{c_k}}}{{\sqrt {2 - {c_{k - 1}}} }}} \right) > 0
$$
and
$$\frac{{{c_k}}}{{\sqrt {2 - {c_{k - 1}}} }} > \left\lfloor {\frac{{{c_k}}}{{\sqrt {2 - {c_{k - 1}}} }}} \right\rfloor , \qquad k = \left\{ {2,3,4, \ldots } \right\},
$$
This implies
$$
\arctan \left( {\frac{{\sqrt {2 - {c_{k - 1}}} }}{{{c_k}}}} \right) < \arctan \left( {1/\left\lfloor {\frac{{{c_k}}}{{\sqrt {2 - {c_{k - 1}}} }}} \right\rfloor } \right)
$$
or
$$
\arctan \left( {\frac{{\sqrt {2 - {c_{k - 1}}} }}{{{c_k}}}} \right) < \arctan \left( {\frac{1}{{{\beta _1}}}} \right)
$$
or
\begin{equation}\label{eq_20}
{2^{k - 1}}\arctan \left( {\frac{{\sqrt {2 - {c_{k - 1}}} }}{{{c_k}}}} \right) < {2^{k - 1}}\arctan \left( {\frac{1}{{{\beta _1}}}} \right).
\end{equation}
We make \eqref{eq_20} into an equality by adding a negative error term  $\varepsilon$ such that
$$
{2^{k - 1}}\arctan \left( {\frac{{\sqrt {2 - {c_{k - 1}}} }}{{{c_k}}}} \right) = {2^{k - 1}}\arctan \left( {\frac{1}{{{\beta _1}}}} \right) + \varepsilon.
$$
Defining the constant $\beta_2$ in accordance with equation \eqref{eq_9}, we find the error term to be
\begin{equation}\label{eq_21}
\varepsilon  = \arctan \left( {\frac{1}{{{\beta _2}}}} \right).
\end{equation}
Since $\varepsilon$ is negative, the constant $\beta_2$ is also negative.

\end{proof}

\section{Iteration methods}

\subsection{Arctangent function}

Since in equation \eqref{eq_7} the constant ${\beta _1}$ is an integer, the first arctangent function term can be computed by any existing method. For example, we can use Euler's formula for the arctangent function
\begin{equation}\label{eq_22}
\arctan \left( x \right) = \sum\limits_{m = 0}^\infty  {\frac{{{2^{2m}}{{\left( {m!} \right)}^2}}}{{\left( {2m + 1} \right)!}}\frac{{{x^{2m + 1}}}}{{{{\left( {1 + {x^2}} \right)}^{m + 1}}}}}.
\end{equation}
Chien-Lih used this formula to develop his iteration method for generating the two-term Machin-like formula for $\pi$ \cite{Chien-Lih2004} and later he found an elegant derivation of this formula \cite{Chien-Lih2005}.
We derived another series expansion of the arctangent function \cite{Abrarov2017b, Abrarov2018a}:
\begin{equation}\label{eq_23}
\arctan \left( x \right) = i\sum\limits_{m = 1}^\infty  {\frac{1}{{2m - 1}}} \left( {\frac{1}{{{{\left( {1 + 2i/x} \right)}^{2m - 1}}}} - \frac{1}{{{{\left( {1 - 2i/x} \right)}^{2m - 1}}}}} \right).
\end{equation}

It interesting that generalizing the derivation method that was used to get equation \eqref{eq_23}, we can find by induction the identity
\[
\arctan\left(x\right) = \sum_{m = 1}^M \arctan\left(\frac{M x}{M^2 + \left(m - 1\right) m x^2}\right),
\]
where $M$ is the order of arctan function expansion; that yields simple approximations like
\[
\arctan\left(x\right) \approx \sum_{m = 1}^M \frac{M x}{M^2 + \left(m - 1\right) m x^2}
\]
and
\[
\arctan\left(x\right) \approx \sum_{m = 1}^M \left(\frac{M x}{M^2 + \left(m - 1\right) m x^2} - \frac{1}{3}\left(\frac{M x}{M^2 + \left(m - 1\right) m x^2}\right)^3\right)
\]
since $\arctan\left(x\right) = x-x^3/3+O\left(x^5\right)$.

The representation \eqref{eq_23} of the arctangent function is not optimal for algorithmic implementation, since it deals with complex numbers. Fortunately, as we showed in \cite{Abrarov2017a}, this series expansion can be significantly simplified to
\begin{equation}\label{eq_24}
\arctan \left( x \right) = 2\sum\limits_{m = 1}^\infty  {\frac{1}{{2m - 1}}\frac{{{g_m}\left( x \right)}}{{g_m^2\left( x \right) + h_m^2\left( x \right)}}},
\end{equation}
where the expansion coefficients are computed by iteration:
$$
{g_1}\left( x \right) = 2/x,\,\,\,{h_1}\left( x \right) = 1,
$$
$$
{g_m}\left( x \right) = {g_{m - 1}}\left( x \right)\left( {1 - 4/{x^2}} \right) + 4{h_{m - 1}}\left( x \right)/x,
$$
$$
{h_m}\left( x \right) = {h_{m - 1}}\left( x \right)\left( {1 - 4/{x^2}} \right) - 4{g_{m - 1}}\left( x \right)/x.
$$

Both series expansions \eqref{eq_22} and \eqref{eq_24} converge rapidly and need no undesirable irrational numbers to compute $\pi$. However, the computational test we performed shows that the series expansion \eqref{eq_24} converges more rapidly by many orders of magnitude than Euler's formula \eqref{eq_22} (see Figures 2 and 3 in \cite{Abrarov2017a}). Therefore, the series expansion \eqref{eq_24} is more advantageous and can be taken to compute the first arctangent function term from the two-term Machin-like formula \eqref{eq_7} for $\pi$.

The second arctangent function term in equation \eqref{eq_7} should not be computed by straightforward substitution of the constant $\beta_2$ into equation \eqref{eq_24}. As mentioned, computing with a ratio of numbers with many digits should be avoided. Instead, the second arctangent function term in equation \eqref{eq_7} can be computed by Newton–Raphson iteration.

\subsection[]{Rational $\beta_2$}

Once the value of the integer $k$ is chosen, it is not difficult to determine the integer ${\beta _1}$ by using equation \eqref{eq_8} with the help of Mathematica. However, determining the second constant $\beta_2$ with \eqref{eq_9} is very hard with increasing $k$, as already mentioned. To overcome this, we proposed a different method \cite{Abrarov2017a}. We define the very simple two-step iteration for $n = \left\{ {2,3,4, \ldots ,k} \right\}$,
$$
\left\{ \begin{aligned}
&{u_n} = u_{n - 1}^2 - v_{n - 1}^2 \hfill \\
&{v_n} = 2{u_{n - 1}}{v_{n - 1}},
\end{aligned}  \right.
$$
where
$$
{u_1} = \frac{{\beta _1^2 - 1}}{{\beta _1^2 + 1}},
$$
and
$$
{v_1} = \frac{{2{\beta _1}}}{{\beta _1^2 + 1}}.
$$
Then
\begin{equation}\label{eq_25}
{\beta _2} = \frac{{2{u_k}}}{{u_k^2 + {{\left( {{v_k} - 1} \right)}^2}}}.
\end{equation}

For $k > 3$, the second arctangent term in \eqref{eq_7} deals only with a rational number ${\beta _2}$. As $k$ increases, the number of digits in the numerator and denominator of the constant ${\beta _2}$ increases. For example, at $k = 6$, equation \eqref{eq_8} yields
$$
{\beta _1} = \left\lfloor {\frac{{\sqrt {2 + \sqrt {2 + \sqrt {2 + \sqrt {2 + \sqrt {2 + \sqrt 2 } } } } } }}{{\sqrt {2 - \sqrt {2 + \sqrt {2 + \sqrt {2 + \sqrt {2 + \sqrt 2 } } } } } }}} \right\rfloor  = 40
$$
and using \eqref{eq_25} we find
$$
\begin{aligned}
\beta _2 &=  - \frac{2634699316100146880926635665506082395762836079845121}{38035138859000075702655846657186322249216830232319} \\
&= - 69.27013796024857670135...\;.
\end{aligned}
$$
Consequently, the two-term Machin-like formula \eqref{eq_7} for $\pi$ is generated as
\small
$$
\begin{aligned}
\frac{\pi }{4} =& 32\arctan{} \left( {\frac{1}{{40}}} \right) \\ 
&- \arctan \left(\frac{38035138859000075702655846657186322249216830232319}{2634699316100146880926635665506082395762836079845121}\right).
\end{aligned}
$$
\normalsize
Lehmer's measure for this two-term Machin-like formula for $\pi$ is $\mu  \approx 1.16751$. However, if we take $k = 27$, then
$$
\beta_1 = \left\lfloor{\frac{{\overbrace {\sqrt {2 + \sqrt {2 + \sqrt {2 +  \ldots  + \sqrt 2 } } } }^{27\,\,{\rm{square}}\,{\rm{roots}}}}}{{\sqrt {2 - \underbrace {\sqrt {2 + \sqrt {2 +  \ldots  + \sqrt 2 } } }_{26\,\,{\rm{square}}\,{\rm{roots}}}} }}} \right\rfloor  = 85445659
$$
and using \eqref{eq_25} we get
$$
\begin{aligned}
\beta_2 &= -\frac{\overbrace{2368557598 \ldots 9903554561}^{522,185,816 \,\, \rm{digits}}}{\underbrace{9732933578 \ldots 4975692799}_{522,185,807 \,\, \rm{digits}}} \\
&=  - 2.43354953523904089818 \ldots  \times 10^8.
\end{aligned}
$$
The corresponding two-term Machin-like formula for $\pi$ is
\small
$$
\frac{\pi}{4} = 67108864\arctan \left(\frac{1}{85445659}\right) - \arctan\left(\frac{\overbrace{ 9732933578 \ldots 4975692799 }^{522,185,807 \,\, \rm{digits}}}{\underbrace{2368557598 \ldots 9903554561}_{522,185,816 \,\, \rm{digits}}}
\right)
$$
\normalsize
for which Lehmer's measure is only $\mu \approx 0.245319$. Such a large number of digits in the numerator and denominator in the second arctangent function may look unusual. However, some formulas for $\pi$ obtained from the Borwein integrals involving the sinc function can also result in ratios of integers with a large number of digits. For example, Bäsel and Baillie reported a formula for $\pi$ that uses a quotient with $453,130,145$ digits in the numerator and $453,237,170$ digits in the denominator \cite{Basel2005}; you can download a file with all the digits of the constant $\beta_2$ from \cite{Abrarov2017c}.

As we can see from these examples, Lehmer's measure decreases with increasing $k$. However, that occurs simultaneously with a rapid increase in the number of digits in the numerator and denominator of the constant $\beta_2$. As a result, taking powers of such fractions becomes very slow. This raises the doubt that Lehmer's measure \eqref{eq_5} is indeed relevant for a given Machin-like formula for $\pi$ when at least one coefficient $\beta_j$ is rational.

To resolve this problem, we considered applying the Newton–Raphson iteration \cite{Abrarov2018b}. Specifically, we showed that each consecutive iteration doubles the number of correct digits in the second term of the arctangent function in equation \eqref{eq_7}. This method is based on the iteration formula:
\begin{equation}\label{eq_26}
{y_{p + 1}} = {y_p} - \left( {1 - {{\sin }^2}\left( {{y_p}} \right)} \right)\left( {\tan \left( {{y_p}} \right) - \frac{1}{{{\beta _2}}}} \right),
\end{equation}
such that
$$
\mathop {\lim }\limits_{n \to \infty } {y_p} = \arctan \left( {\frac{1}{{{\beta _2}}}} \right).
$$
The most important advantage of \eqref{eq_26} is that the rational number $1/{\beta _2}$ is not involved in the computation of the trigonometric functions. As we can see from \eqref{eq_26}, $1/{\beta _2}$ is no longer problematic because it is not taken to a power, nor is it the argument of sine or tangent, which consumes most of the runtime. Instead, it is only applied in a single subtraction. This single subtraction (that can be implemented by changing the precision) takes a negligibly small amount of time as compared to the time to compute $\sin \left( {{y_p}} \right)$ or $\tan \left( {{y_p}} \right)$.

To reduce the number of trigonometric functions from two to one, it is convenient to put equation (26) into the form
\begin{equation}\label{eq_27}
{y_{p + 1}} = {y_p} - \left( {1 - {{\left( {\frac{{2\tan \left( {\frac{{{y_p}}}{2}} \right)}}{{1 + {{\tan }^2}\left( {\frac{{{y_p}}}{2}} \right)}}} \right)}^2}} \right)\left( {\frac{{2\tan \left( {\frac{{{y_p}}}{2}} \right)}}{{1 - {{\tan }^2}\left( {\frac{{{y_p}}}{2}} \right)}} - \frac{1}{{{\beta _2}}}} \right),
\end{equation}
using the elementary trigonometric identities
$$
\sin \left( {{x}} \right) = \frac{{2\tan \left( {{{{x}}}/{2}} \right)}}{{1 + {{\tan }^2}\left( {{{{x}}}/{2}} \right)}}
$$
and
$$
\tan \left( {{x}} \right) = \frac{{2\tan \left( {{{{x}}}/{2}} \right)}}{{1 - {{\tan }^2}\left( {{{{x}}}/{2}} \right)}}.
$$

The tangent function can be found, for example, by using the equation
$$
\tan \left( x \right) = \frac{{\sin \left( x \right)}}{{\cos \left( x \right)}} = \frac{{\sum\limits_{n  = 0}^\infty  {\frac{{{{\left( { - 1} \right)}^n }}}{{\left( {2n  + 1} \right)!}}{{ x }^{2n  + 1}}} }}{{\sum\limits_{n  = 0}^\infty  {\frac{{{{\left( { - 1} \right)}^n }}}{{\left( {2n } \right)!}}{{ x }^{2n }}} }}
$$
representing the ratio based on the Maclaurin series expansions for the sine and cosine functions. Alternatively, we can use the series expansion
$$
\tan \left( x \right) = \sum\limits_{n  = 1}^\infty  {\frac{{{{\left( { - 1} \right)}^{n  - 1}}{2^{2n }}\left( {{2^{2n }} - 1} \right){B_{2n }}}}{{\left( {2n } \right)!}}{{ x }^{2n  - 1}}},
$$
where ${B_{2n }}$ is a Bernoulli number. There are several other ways to compute the tangent, like continued fractions \cite{Havil2012, Trott2011} or Newton–Raphson iteration again \cite{Muller2016}. Perhaps the argument reduction method for the tangent function can also be used to improve accuracy, but we did not implement that, to keep the algorithm as simple as possible.

\subsection{Tangent function}

Equation \eqref{eq_27} is based on the Newton–Raphson iteration method. However, this iteration formula contains the tangent function. We showed \cite{Abrarov2018b} that once the first arctangent term is computed, the number of correct digits in $\pi$ can be doubled at each consecutive step of the iteration as with the Newton–Raphson iteration method. To approximate the tangent function with high accuracy, we can apply the Newton–Raphson iteration again \cite{Muller2016}. The derivation of the iteration-based equation for the tangent function is not difficult. Let $\gamma =\arctan\left(z\right)$.

Using the Newton–Raphson iteration formula
$$
{z_{q+1}} = {z_{q}} - \frac{{f\left( {{z_{q}}} \right)}}{{f'\left( {{z_{q}}} \right)}},
$$
where
$$
f\left( z \right) = \arctan \left( z \right) - \gamma \;\text{and}\; f'\left( z \right) = \frac{d}{{dy}}\left( {\arctan \left( z \right) - \gamma} \right) = \frac{1}{{1 + {z^2}}},
$$
we get
\begin{equation}\label{eq_28}
{z_{q+1}} = {z_{q}} - \left( {1 + z_{q}^2} \right)\left( {\arctan \left( {{z_{q}}} \right) - \gamma} \right)
\end{equation}
such that
$$
\tan \left( \gamma \right) = \mathop {\lim }\limits_{q \to \infty } {z_q}.
$$

Substituting $\gamma = {y_p}/2$ into \eqref{eq_28} yields
\begin{equation}\label{eq_29}
{z_{q+1}} = {z_{q}} - \left( {1 + z_{q}^2} \right)\left( {\arctan \left( {{z_{q}}} \right) - \frac{{{y_p}}}{2}}\right).
\end{equation}
The iteration-based expansion of the series \eqref{eq_24} for the arctangent function converges very rapidly. Therefore, we can apply it to compute the arctangent function in equation \eqref{eq_29}.

\section{Implementation}

\subsection{Linear convergence}

The series expansion \eqref{eq_24} of the arctangent function can be computed as follows.
\small
\begin{verbatim}
atan[M_,x_]:=atan[M,x]=2*Sum[1/(2*m-1)*g[m,x]/(g[m,x]^2+
    h[m,x]^2),{m, M}]

g[1,x_]:=g[1,x]=2/x;
g[m_,x_]:=g[m,x]=g[m-1,x]*(1-4/x^2)+4/x*h[m-1,x]
h[1,_] = 1;
h[m_,x_]:=h[m,x]=h[m-1,x]*(1-4/x^2)-4/x*g[m-1, x]
\end{verbatim}
\normalsize

Next we define the nested radicals consisting of square roots of $2$
\small
\begin{verbatim}
c[0]=0;
c[k_]:=Sqrt[2+c[k-1]]
\end{verbatim}
\normalsize

This computes the constants $\beta_1$ and $\beta_2$ for the two-term Machin-like formula \eqref{eq_7} for $\pi$ at $k=6$.
\small
\begin{verbatim}
\[Beta]1=Floor[c[6]/Sqrt[2-c[6-1]]]
\end{verbatim}
\normalsize
$\qquad\quad 40$ \\
For the constant $\beta_2$, we use the iteration-based formula \eqref{eq_25} instead of equation \eqref{eq_9}.
\small
\begin{verbatim}
u[1]=(\[Beta]1^2-1)/(\[Beta]1^2+1)
\end{verbatim}
\Large
$\qquad \frac{1599}{1601}$

\small
\begin{verbatim}
u[n_]:=u[n-1]^2-v[n-1]^2

v[1]=2*\[Beta]1/(\[Beta]1^2+1)
\end{verbatim}
\Large
$\qquad \frac{80}{1601}$

\small
\begin{verbatim}
v[n_]:=2*u[n-1]*v[n-1]

\[Beta]2=2*u[6]/(u[6]^2+(v[6]-1)^2)
\end{verbatim}
\Large
$\qquad -\frac{2634699316100146880926635665506082395762836079845121}{38035138859000075702655846657186322249216830232319}$
\normalsize
\vspace{0.5cm}

Define a function for the Lehmer's measure corresponding to a two-term Machin-like formula \eqref{eq_7} for $\pi$.
\small
\begin{verbatim}
LehmerMeasure[\[Beta]1_,\[Beta]2_]:=1/Log10[Abs[\[Beta]1]]+
    1/Log10[Abs[\[Beta]2]]
\end{verbatim}
\normalsize
Here is Lehmer's measure for the case $k = 6$.
\small
\begin{verbatim}
LehmerMeasure[\[Beta]1,\[Beta]2]//N
\end{verbatim}
\normalsize
$\qquad\quad 1.16751$
\vspace{0.5cm}

The accuracy improves with each iteration, so we do not need to use the highest accuracy at each step. At $k=6$, the Newton–Raphson iteration-based formula \eqref{eq_29} gives $4$ to $5$ correct digits of the tangent function at each step. Therefore, it is reasonable to use the argument $5q+2$ in {\ttfamily{SetPrecision}}, where $2$ is taken to minimize rounding or truncation errors.

Recall that $\lim\limits_{p\to\infty}y_p=\arctan\left(1/\beta_2\right)$. As an initial guess for the Newton–Raphson iteration, we choose $\arctan\left(1/\beta_2\right)$, since this is close to the actual value of the tangent function $\tan\left(y_p/2\right)$.

This sets up the recurrence for $z_q$ as defined in \eqref{eq_29}.
\small
\begin{verbatim}
z[1,_]=SetPrecision[ArcTan[1/\[Beta]2],5]
\end{verbatim}
\normalsize
$\qquad -0.014435$

\small
\begin{verbatim}
z[q_,x_]:=z[q,x]=SetPrecision[z[q-1,x]-(1+z[q-1,x]^2)*
    (atan[q,z[q-1,x]]-x),5*q+2]
\end{verbatim}
\normalsize
This part of the program, which computes the tangent function, takes most of the runtime. It uses the Newton–Raphson iteration built on the basis of the series expansion \eqref{eq_24} of the arctangent function (see \eqref{eq_29}).

This sets up the Newton–Raphson iteration formula \eqref{eq_27}.
\small
\begin{verbatim}
y[1]=1/\[Beta]2
\end{verbatim}
\Large
$\qquad -\frac{38 035 138 859 000 075 702 655 846 657 186 322 249 216 830 232 319}
{2 634 699 316 100 146 880 926 635 665 506 082 395 762 836 079 845 121}$
\vspace{0.5cm}

\small
\begin{verbatim}
y[p_]:=y[p]=With[{a=2*z[p,1/2*y[p-1]],b=z[p,1/2*y[p-1]]^2},
    SetPrecision[y[p-1],5*p+2]-(1-(a/(1+b))^2)*
		    (a/(1-b)-1/\[Beta]2)]
\end{verbatim}
\normalsize
The next part of the program invokes the value $\tan\left(y_p/2\right)$ and performs just a few arithmetic operations. As mentioned, the numerator and denominator of $1/\beta_2$ contain many digits, but they are not involved in computing the tangent function. Only a minor part of the time is needed to subtract this number (see equation \eqref{eq_27}). Consequently, Lehmer's measure applies unconditionally.

The table shows how the values of the arctangent function gain digits at each iteration step.
\small
\begin{verbatim}
Table[{p-1,y[p]},{p,2,11}]//TableForm
\end{verbatim}
\footnotesize
$$
\begin{tabular}{ll}
 1  & -0.0144354827911 \\
 2  & -0.014435232407997704 \\
 3  & -0.01443523240799679443925 \\
 4  & -0.0144352324079967944392951115 \\
 5  & -0.014435232407996794439295110969614 \\
 6  & -0.01443523240799679443929511096961893161 \\
 7  & -0.0144352324079967944392951109696189315443963 \\
 8  & -0.014435232407996794439295110969618931544397010224 \\
 9  & -0.01443523240799679443929511096961893154439701021536520 \\
 10 & -0.0144352324079967944392951109696189315443970102153652922241
\end{tabular}
$$
\normalsize

This shows the linear rate of convergence to $\pi$ in terms of the number of accurate digits.
\small
\begin{verbatim}
Table[Abs[MantissaExponent[N[Pi-4*(2^(6-1)*atan[1000,
    1/\[Beta]1]+y[p]),75]][[2]]],{p,2,11}]
\end{verbatim}
\normalsize
$\qquad \text{\{5, 12, 21, 26, 31, 36, 41, 46, 51, 56\}}$
\vspace{0.5cm}
\small
\begin{verbatim}
ListLinePlot@%
\end{verbatim}
\normalsize

\begin{figure}[H]
\begin{center}
\includegraphics[width=20pc]{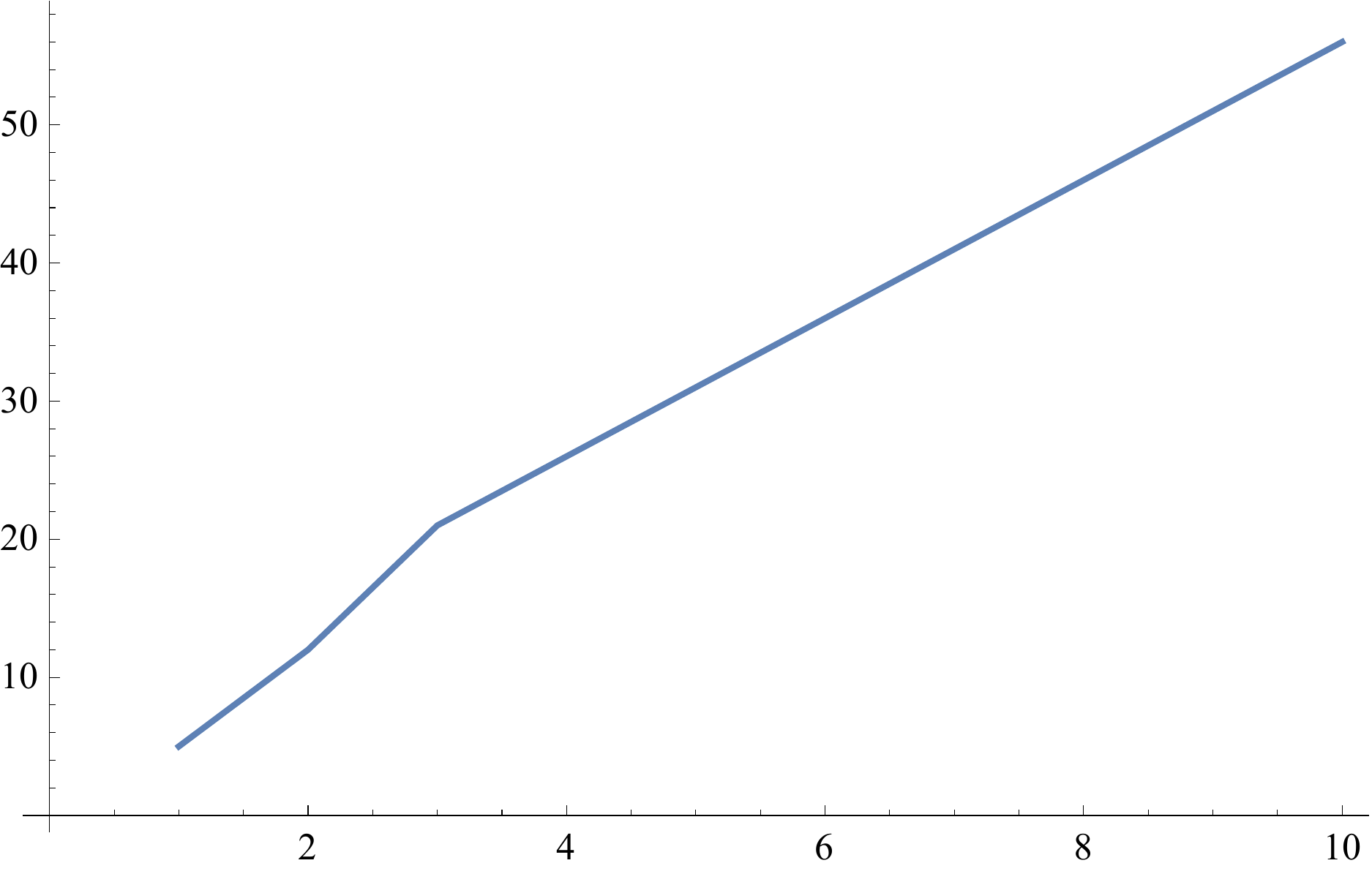}\hspace{2pc}%
\begin{minipage}[b]{22pc}
\vspace{0.3cm}
\end{minipage}
\end{center}
\end{figure}
\noindent After the first two iterations, each of the following iterations adds five correct digits to the approximation of $\pi$.
\small
\begin{verbatim}
Differences@%%
\end{verbatim}
\normalsize
$\qquad \text{\{7, 9, 5, 5, 5, 5, 5, 5, 5\}}$
\vspace{0.5cm}

The convergence rate increases as Lehmer's measure decreases, which can be readily confirmed by increasing $k$ and readjusting the {\ttfamily SetPrecision} parameter in this algorithm. No undesirable irrational numbers are needed. Furthermore, since Lehmer's measure can be made vanishingly small, there is no upper bound in the convergence rate per iteration.

\subsection{Quadratic convergence}

Consider a variation of the algorithm based on the Newton–Raphson iterations for the tangent function that can be implemented to get quadratic convergence to $\pi$. Assume that only $50$ decimal digits of $\pi$ are known at the beginning.
\small
\begin{verbatim}
d = 50;

Y=Floor[SetPrecision[Pi/4-32*atan[Ceiling[2*d/4],1/40],
    2*d]*10^d]/10^d
\end{verbatim}
\Large
\vspace{0.5cm}
$\qquad -\frac{1 443 523 240 799 679 443 929 511 096 961 893 154 439 701 021 537}
{100 000 000 000 000 000 000 000 000 000 000 000 000 000 000 000 000}$
\normalsize
\vspace{0.5cm}

We determined experimentally that at $k=6$, equation \eqref{eq_24} provides four or five correct digits of $\pi$ at each successive step. Therefore, it is sufficient to take $d/4$ terms. However, to exclude truncation and rounding errors, we use $d/2$ terms of the arctangent function.

We define {\ttfamily acc[q]} to determine the accuracy of computation and for the number of summation terms in approximation \eqref{eq_24} of the arctangent function.
\small
\begin{verbatim}
acc[q_]:=If[4*2^q<2*d,4*2^q,2*d+10]
\end{verbatim}
\normalsize

The multiplier $4\times2^2$ was found experimentally. When $4\times2^2>2d$, we restrict its rapid growth to $2d+10$, as we do not need extra accuracy at this stage. We add $10$ to exclude truncation and rounding errors. We have taken the initial value $z = -0.007$.

We use equation \eqref{eq_29}.
\small
\begin{verbatim}
Z[1]:=-0.007;
Z[q_]:=Z[q]=SetPrecision[Z[q-1]-(1+Z[q-1]^2)*
    (atan[Ceiling[acc[q-1]/4],Z[q-1]]-Y/2),acc[q]]
\end{verbatim}
\normalsize

For our choice of $d$, here is the required index.
\small
\begin{verbatim}
q=1;While[acc[q]<2*d,q++];q++;q
\end{verbatim}
\normalsize
$\qquad 6$
\vspace{0.5cm}

Here are the first six values of {\ttfamily Z}.
\small
\begin{verbatim}
Table[Z[q],{q,7}]//Column
\end{verbatim}
\small
\[
\begin{split}
-&0.007 \\
-&0.007217741202976670 \\
-&0.0072177415380702764955539959679745 \\
-&0.0072177415380702764963413557275658146208551067226428615\ddots \\
 &42089449654 \\
-&0.0072177415380702764963413557275658146208547604442639281\ddots \\
 &57473637855747694358873239684119495265539483149682234585\ddots \\
 &0 \\
\end{split}
\]
\[
\begin{split}
-&0.0072177415380702764963413557275658146208547604442639281\ddots \\
 &57473637855747694358917310352398756159236332112742091353\ddots \\
 &7 \\
-&0.0072177415380702764963413557275658146208547604442639281\ddots \\
 &57473637855747694358917310352398756159236332112742091353\ddots \\
 &7
\end{split}
\]
\normalsize

Once the tangent function is computed, we substitute it into equation \eqref{eq_27}. The value {\ttfamily Y2} has double the accuracy of {\ttfamily Y} with only one iteration.
\small
\begin{verbatim}
Y2=Y-(1-(2*Z[q]/(1+Z[q]^2))^2)*(2*Z[q]/(1-Z[q]^2)-1/\[Beta]2)
\end{verbatim}
\normalsize
\[
\begin{split}
-&0.01443523240799679443929511096961893154439701021536529222\ddots \\
 &2928740206431337484883800802495521020607068736057611275
\end{split}
\]

The approximate values of the arctangent function by iteration and by the built-in function {\ttfamily ArcTan} match to $100$ places.
\small
\begin{verbatim}
N[Y2,2*d]
N[ArcTan[1/\[Beta]2],2*d]
\end{verbatim}
\normalsize
\[
\begin{split}
-&0.01443523240799679443929511096961893154439701021536529222\ddots \\
 &292874020643133748488380080249552102060706874 \\
-&0.01443523240799679443929511096961893154439701021536529222\ddots \\
 &292874020643133748488380080249552102060706874
\end{split}
\]

Since we got {\ttfamily Y2} with double the accuracy, it can now be used to compute $\pi$ with significantly improved accuracy.
\small
\begin{verbatim}
pi2=4*(32*atan[Ceiling[4*d/4],1/40]+Y2)
\end{verbatim}
\normalsize
\[
\begin{split}
&3.141592653589793238462643383279502884197169399375105820974\ddots \\
&944592307816406286208998628034825342117067980868275870
\end{split}
\]
This shows that the number {\ttfamily d2} of correct decimal digits of {\ttfamily pi2} doubled from $50$ to $101$.
\small
\begin{verbatim}
d2=Abs[MantissaExponent[N[Pi-pi2,10000]][[2]]]
\end{verbatim}
\normalsize
$\qquad 101$
\vspace{0.5cm}

\noindent More directly, this shows the complete match between the computed approximation of $\pi$ and that provided by Mathematica.
\small
\begin{verbatim}
SetPrecision[pi2,d2]
\end{verbatim}
\normalsize
\[
\begin{split}
&3.141592653589793238462643383279502884197169399375105820974\ddots \\
&9445923078164062862089986280348253421170680
\end{split}
\]
\small
\begin{verbatim}
N[Pi,d2]
\end{verbatim}
\normalsize
\[
\begin{split}
&3.141592653589793238462643383279502884197169399375105820974\ddots \\
&9445923078164062862089986280348253421170680
\end{split}
\]

The first arctangent function in the two-term Machin-like formula \eqref{eq_7} for $\pi$ can also be found by using the same algorithm based on the Newton–Raphson iteration. Consequently, this method results in quadratic convergence to $\pi$. However, unlike the Brent–Salamin algorithm (also known as the Gauss–Brent–Salamin algorithm) with quadratic convergence to $\pi$ \cite{Berggren2004}, our approach does not involve any irrational numbers. The number of summation terms in equation \eqref{eq_24} and the number of iteration cycles for computation of the tangent function \eqref{eq_29} decrease with increasing $k$. This can be confirmed by using the code given here. To the best of our knowledge, this is the first algorithm showing the feasibility of quadratic convergence to $\pi$ without using any irrational numbers.

\section{Conclusion}

In this article, we presented a new algorithm to compute the two-term Machin-like formula \eqref{eq_7} for $\pi$ and showed an example where the condition $\beta_j\in \mathbb{Z}$ was not necessary in order to validate Lehmer's measure \eqref{eq_5}. Since this algorithmic implementation lets us avoid subsequent exponentiation of the second constant $\beta_2$, this approach may be promising for computing $\pi$ more rapidly without using irrational numbers.

\section*{Acknowledgments}

This work is supported by National Research Council Canada, Thoth Technology, Inc., York University, Epic College of Technology and Epic Climate Green (ECG) Inc. The authors wish to thank the reviewer for useful comments and recommendations. Constructive suggestions from the editor that improved the content of this work are greatly appreciated.

\bigskip

\end{document}